\documentclass[12pt,reqno]{amsart}
\usepackage{amssymb,delarray}
\usepackage{amsfonts,amscd}
\usepackage{epsfig}
\usepackage[all]{xy}

\def\leq{\leqslant}
\def\geq{\geqslant}

\newtheorem{thm}{Theorem}
\newtheorem{lem}
{Lemma}
\newtheorem{prop}
{Proposition}
\newtheorem{claim}
{Claim}
\newtheorem{df}
{Definition}
{Corollary}
{Remark}
{Question}
\newtheorem{ex-thm}{Theorem-Example}

{\catcode`\@=11
\gdef\n@te#1#2{\leavevmode\vadjust{%
 {\setbox\z@\hbox to\z@{\strut#1}%
  \setbox\z@\hbox{\raise\dp\strutbox\box\z@}\ht\z@=\z@\dp\z@=\z@%
  #2\box\z@}}}
\gdef\leftnote#1{\n@te{\hss#1\quad}{}}
\gdef\rightnote#1{\n@te{\quad\kern-\leftskip#1\hss}{\moveright\hsize}}
\gdef\?{\FN@\qumark}
\gdef\qumark{\ifx\next"\DN@"##1"{\leftnote{\rm##1}}\else
 \DN@{\leftnote{\rm??}}\fi{\rm??}\next@}}

\begin{document}
\baselineskip=13.7pt plus 2pt

\title[On finite covers of projective spaces] {On finite covers of projective spaces, generic outside of codimension three}

\author[Vik.S. Kulikov]{Vik.S. Kulikov}

\address{Steklov Mathematical Institute of Russian Academy of Sciences, Moscow, Russia}
 \email{kulikov@mi-ras.ru}

\dedicatory{} \subjclass{}
\thanks{This work was performed at the Steklov International Mathematical Center and supported by the Ministry of Science and Higher Education of the Russian Federation (agreement no. 075-15-2022-265).}

\keywords{}

\maketitle 

\begin{abstract} In the article, we generalize to an arbitrary dimension  results obtained by the author and S.Yu. Nemirovski 
that related to Chisini's conjecture claimed that if $B$ is the discriminant of two generic covers $f_i:X_i\to \mathbb P^2$, $i=1,2$, and $\max (\deg f_1,\deg f_2) \geq 5$, then $f_1$ and $f_2$ are equivalent covers. 
\end{abstract} \vspace{0.5cm}

\def\st{{\sf st}}

\def\st{{\sf st}}

\setcounter{section}{-1}

\section{Introduction}
Let $X$ be a non-singular irreducible projective variety, $\dim X=k$,  defined over the field of complex numbers $\mathbb C$, and 
$f: X \to \mathbb P^k$  a finite morphism onto the projective space $\mathbb P^k$, $\deg f =d$. Two finite morphisms $f_i:X_i\to \mathbb P^k$, $i=1,2$, are {\it equivalent} if there is an isomorphism $\varphi: X_1\to X_2$ such that $f_1=f_2\circ\varphi$.

The {\it discriminant} $ B_f=\{ p\in \mathbb P^k\, |\, \sharp(f^{-1}(p))\leq d-1\}$ of a finite morphism  $f:X\to \mathbb P^k$ is a subvariety of 
$\mathbb P^k$ of codimension 1. The map  $f:X\setminus f^{-1}(B_f) \to \mathbb P^k\setminus B_f$ is an unramified cover of degree $d=\deg f$.
The divisor   $f^*(B_f)\in \text{Div}\, X$ is a sum of divisors, 
$$f^*(B_f) = \sum r_i R_i+C_f,\qquad r_i\geq 2,$$ 
where ${ R_f}=\sum R_i$ (called the {\it ramification locus} of $f$) and $C_f$ are reduced divisors. 

The set of {\it generic covers} of projective spaces, as a subset of all finite covers, plays an important role in the study of the properties of nonsingular projective varieties. In the case of $\dim X=1$, a finite morphism $f: X \to \mathbb P^1$, $\deg f=d$, is called a {\it generic cover} of the projective line $\mathbb P^1$ if the number $\sharp(f^{-1}(p))\geq d-1$ for all $p\in\mathbb P^1$. 
In the case of $\dim X=2$, a finite morphism $f: X \to \mathbb P^2$ is a {\it generic cover} of the projective plane if 
\begin{itemize}
\item[$(\text{i})$] $f^*(B_f)= 2R_f+C_f$, 
\item[$(\text{ii})$] the curve $B_f\subset \mathbb P^2$ is irreducible and $f_{\mid R_f}: R_f\to B_f$ is a birational morphism, 
\item[$(\text{iii})$]  all points $p\in \text{Sing}\, B_f$ are ordinary nodes and ordinary cusps only. 
\end{itemize}

\begin{df} Let $D\subset X$ be a divisor in a non-sindular variety $X$, $\dim X=k>1$. A  point $p\in D$ is said to be an ordinary $(k-2)$-dimensional node (resp., an ordinary $(k-2)$-dimensional cusp (below simply $(k-2)$-node (resp.,  $(k-2)$-cusp)) if there are a neighborhood (in complex topology) $U_p\subset X$ of the point $p$ and local coordinates $z_1,\dots, z_k$ in $U_p$ such that $D\cap U_p$ is given by equation $z_1z_2=0$ (resp., $z_1^2-z_2^3=0$). 
\end{df} 
 
A finite morphism $f: X \to \mathbb P^k$, $k\geq 2$, is said to be a {\it cover, generic outside of codimension three} if 
\begin{itemize} 
\item[(i)] $f^{*}(B_f)=2R_f+C_f$,  
\item[(ii)] $B_f$ is irreducible and $f_{|R_f}:R_f\to B_f$ is a birational morphism;  
\item[(iii)] $\text{Sing}\, B_f = S_{f,n}\sqcup S_{f,c}\sqcup S_{f,3}\subset B_f$ is a disjoint union of three quasi-projective varieties (if $\text{Sing}\, B_f\neq \emptyset$), where
\begin{itemize} \item[(1)] $S_{f,3}$ is a closed subvariety of $\text{Sing}\, B_f$, $\dim S_{f,3}\leq k-3$, 
\item[(2)] $S_{f,2}=S_{f,n}\sqcup S_{f,c}$ is an everywhere dense Zariski open subset of $\text{Sing}\, B_f$, $S_{f,n}$ is the set of $(k-2)$-nodes, and  $S_{f,c}$ is the set of $(k-2)$-cusps of $\text{Sing}\, B_f$.
\end{itemize}  
\end{itemize}

We call $S_{f,n}$ the {\it nodal part} of $\text{Sing}\, B_f$ and $S_{f,c}$ is called the {\it cuspidal part} of $\text{Sing}\, B_f$. 

\begin{thm} \label{gdeg} Let $f: X \to \mathbb P^k$ be a cover, generic outside of codimension three. 

If $k>1$ and $S_{f,c}=\emptyset$ then the discriminant $B_f$ is non-singular and $\deg f=2$. 

If $S_{f,n}=\emptyset$, but $\text{Sing}\, B_f\neq \emptyset$, then  $\deg f\leq 5$.
\end{thm}

Denote by $\text{pr}_{L}:\mathbb P^N\dasharrow \mathbb P^{k}$ the linear projection with center in a linear projective linear subspace $L\subset \mathbb P^N$,
$\dim L=N-k-1$, and let $\text{Gr}_{N,N-k-1}$ be the Grassmannian the points of which parameterize the set of all $(N-k-1)$-dimensional projective subspaces $L$ of the projective space $\mathbb P^N$ (further, the point in Grassmannian corresponding to a linear projective subspace $L\subset \mathbb P^N$ is denoted by the same letter $L$). It is well-known (see, for example, Theorem 7 in Chapter I of \cite{Sh}) that if $L\in \text{Gr}_{N,N-k-1}$  is such that $X\cap L=\emptyset$, then  the restriction $f_L=\text{pr}_{L|X}:X\to\mathbb P^k$ of the projection $\text{pr}_L$ to $X$, $\dim X=k$, is a finite morphism. 

\begin{thm} \label{main1} Let $X\subset \mathbb P^N$ be a non-singular irreducible projective varriety, $\dim X=k\geq 1$ and $\deg X=d$. Then there is  $L\in \text{Gr}_{N,N-k-1}$ such that the linear projection $f_L:X\to\mathbb P^k$ is a cover, generic outside of codimension three, $\deg f_L=d$.
\end{thm}

Note that in the case of $k\leq 2$ Theorem \ref{main1} is well-known and in the case of $k=2$ its modern proof  can be found in \cite{C-F}.

The following Theorem is a generalization to the arbitrary dimension of results contained in \cite{K2} and \cite{N} that related to Chisini's conjecture (\cite{Ch}) claimed that if $B$ is the discriminant of generic covers $f_i:X_i\to \mathbb P^2$, $i=1,2$, and $\max (\deg f_1,\deg f_2) \geq 5$, then $f_1$ and $f_2$ are equivalent covers. 

\begin {thm} \label{chis} Let $B\subset \mathbb P^k$ be the discriminant of  two  covers $f_i: X \to \mathbb P^k$, generic outside of codimension three, $i=1,2$, where $k>1$. Assume that either 
\begin{itemize}
\item[(1)] $\max(\deg f_1,\deg f_2)>11$, 
\end{itemize}
or
\begin{itemize}
\item[(2)] $f_1=f_{L}$ is a  linear projection and $\deg f_1\geq 5$. 
\end{itemize}
Then $f_1$ and $f_2$ are equivalent covers.
\end{thm} 

In Section $1$, properties of covers, generic outside of codimension three, are investigated. In Section $2$, we give another inductive definition of covers, generic outside of codimension three, and in Section $3$, Theorems 1 -- 3 are proved.

\section{Properties of covers, generic outside of codimention three} \label{propert}
\subsection{Monodromy of  covers generic outside of codimention three} \label{monod} As in the general case of finite covers of projective spaces, a cover $f:X\to\mathbb P^k$, generic outside of codimension three, induces a {\it monodromy homomorphism} $f_*:\pi_1=\pi_1(\mathbb P^k\setminus B_f,b)\to \mathbb S_{d}$  to the symmetric group $\mathbb S_{d}$ acting  on the fibre $f^{-1}(b)=\{ q_1,\dots,q_{d}\}\subset X$ over a point $b\in\mathbb P^k\setminus B_f$. Its image $G_f:=f_*(\pi_1(\mathbb P^k\setminus B_f,b))\subset \mathbb S_{d}$  is called the {\it monodromy group} of  $f$. Since $X$ is irreducible, then $X\setminus f^{-1}(B_f)$ is connected. Therefore $G_f$ acts transitively on $f^{-1}(b)$. 

It is well-known that $\pi _1$ is generated by so called {\it geometric generators} $\gamma$, i.e., elements represented by positively oriented little loops $\Gamma$ in $\mathbb P^k\setminus B_f$ around $B_f$.  It follows from properties ($\text{i}$) and ($\text{ii}$) in the definition of covers, generic outside of codimension three, that $f_*(\gamma)\in\mathbb S_d$ is a transposition for each geometric generator $\gamma\in\pi_1$.   
Therefore $G_f=\mathbb S_d$, since $G_f$ is a transitive subgroup of $\mathbb S_d$ generated by transpositions. 

The following Proposition shows that the condition of irreducibility of the curve $B_f$ in the definition of generic covers of $\mathbb P^2$ is superfluous.
\begin{prop}\label{irred} Let $f: X \to \mathbb P^2$ be a finite morphism such that  
\begin{itemize}
\item[$(\text{i})$] the singular points of its discriminant $B_f$ are ordinary nodes and ordinary cusps only, 
\item[$(\text{ii})$] $f^*(B_f)= 2R_f+C_f$, 
\item[$(\text{iii})$] for each irreducible component $B_{f,i}$ of the discriminant $B_f$ the restriction 
$$f_{\mid R_{f,i}}: R_{f,i}=R_f\cap f^{-1}(B_{f,i})\to B_{f,i} \subset B_f$$ of $f$ to $R_{f,i}$ is a birational morphism.
\end{itemize}
Then $f$ is a generic cover of the projective plane.
\end{prop}
\proof To prove Proposition, we only need to prove that $B_f$ is an irreducible curve. Assume that $B_f=B_{f,1}\cup B_{f,>1}$ consists of more than one irreducible component $B_{f,1}$, where $B_{f,>1}$ is the closure of $B_{f}\setminus B_{f,1}$ in $\mathbb P^2$. 

As above, it follows from the irreducibility of $X$ and property $(\text{ii})$ that the monodromy group $G_f=f_*(\pi_1(\mathbb P^2\setminus B_f,b))\subset \mathbb S_{\deg f}$ is generated by transpositions and acts transitively on the set $f^{-1}(b)$. Therefore $G_f=\mathbb S_{\deg f}$. 

On the over hand, it follows from property $(\text{i})$ that $B_{f,1}$ and $B_{f,>1}$ meet transversally at non-singular points of $B_{f,1}$ and $B_{f,>1}$. 
Let  $L\subset \mathbb P^2$ be a line meeting transversally the curve  $B_f$ at non-singular points of $B_f$.  

By Theorem 1 in \cite{Ku}, the fundamental group $$\pi_1(\mathbb C^2\setminus B_f,b)=\pi_1(\mathbb C^2\setminus B_{f,1},b)\times \pi_1(\mathbb C^2\setminus B_{f,>1},b)$$ is the direct product of the groups $\pi_1(\mathbb C^2\setminus B_{f,1},b)$ and $\pi_1(\mathbb C^2\setminus B_{f,>1},b)$, where 
$\mathbb C^2=\mathbb P^2\setminus L$.  Therefore the group $G_f=\mathbb S_{\deg f}$ is generated by the elements of its two nontrivial subgroups $G_1$ and $G_{>1}$ such that all elements of $G_1$ commute with all elements of $G_{>1}$, where $G_1$ and $G_{>1}$ are the images of the compositions of the epimorphisms $\pi_1(\mathbb C^2\setminus B_1,b)\to \pi_1(\mathbb P^2\setminus B_1,b)$ and $\pi_1(\mathbb C^2\setminus B_{>1},b)\to \pi_1(\mathbb P^2\setminus B_{>1},b)$ that are defined by imbedding of the affine surface $\mathbb C^2$ in $\mathbb P^2$ and the monodromy homomorphism $f_*$. But,  it is possible only if $\deg f=2$. The case of $\deg f=2$ is impossible also, because $X$ is a non-singular variety and in this case $B_f$ must be a non-singular curve in $\mathbb P^2$ and therefore $B_f$ must be irreducible. \qed

\subsection{Local behavior of covers, generic outside of codimension three, over the points from $\bf S_{f,2}=S_{f,c}\cup S_{f,n}$} \label{loc} 
Let $k\geq 2$. By definition of covers, generic outside of codimension three, for each point $p\in S_{f,2}\subset B_f$, there are a neighborhood $U_p$  and local coordinates $z_1,\dots,z_k$ in $U_p$ such that $z_1(p)=z_2(p)=0$ and $B_f$ is given by equation $z_1z_2 = 0$ if $p\in S_{f,n}$ and 
$z^2_1 = z_2^3$ if $p\in  S_{f,c}$. 

Choose  positive  $\varepsilon\ll 1$ such that 
$$V_{p,\varepsilon}=\{ (z_1,\dots, z_k)\, \mid\, |z_1|+|z_2|< \varepsilon\, \text{and}\, |z_i|< \varepsilon\, \text{for}\, i\geq 3\}\subset U_p.$$
The neighbourhood $V_{p,\varepsilon}=\mathcal B_{\varepsilon}\times\mathcal D_{\varepsilon}^{k-2}$ is the product of a ball 
$$\mathcal B_{\varepsilon}=\{ (z_1, z_k)\in\mathbb C^2\, \mid\, |z_1|+|z_2|< \varepsilon\}$$ 
in $\mathbb C^2$ and a ($k-2$)-disc $\mathcal D_{\varepsilon}^{k-2}=\{ (z_3,\dots, z_k)\in\mathbb C^{k-2}\, \mid\, |z_i|< \varepsilon\}$ in $\mathbb C^{k-2}$ (if $k=2$ then $\mathcal D^0_{\varepsilon}$ is a point). Then $B_f\cap V_{p,\varepsilon}=B_{\varepsilon}\times \mathcal D^{k-2}_{\varepsilon}$ is the product of the cusp $B_{\varepsilon}=\{ (z_1,z_2)\in\mathcal B_{\varepsilon}\, \mid\, z_1^2=z_2^3\}$ if $p\in S_{f,c}$ (resp., the node  
$B_{\varepsilon}=\{ (z_1,z_2)\in\mathcal B_{\varepsilon}\, \mid\, z_1z_2=0\}$ if $p\in S_{f,n}$) and the ($k-2$)-disc $\mathcal D^{k-2}_{\varepsilon}$.

If $p\in S_{f,2}$ then $V_{p,\varepsilon}\setminus B_f=(\mathcal B_{\varepsilon}\setminus B_{\varepsilon})\times \mathcal D_{\varepsilon}$ and it is well known that the local fundamental group $\pi_1^{loc}(B_f,p):=\pi_1(V_{p,\varepsilon}\setminus B_f,b=b'\times b^{\prime\prime})=\pi_1(\mathcal B_{\varepsilon}\setminus B_{\varepsilon},b')$ does not depend on $\varepsilon\ll 1$ and it is generated by two geometric generators $\gamma_1$ and $\gamma_2$. In addition, if $p\in S_{f,c}$, then $\pi_1^{loc}(B_f,p)=\text{Br}_3=\langle \gamma_1,\gamma_2\, \mid\, \gamma_1\gamma_2\gamma_1=\gamma_2\gamma_1\gamma_2\rangle$ is the braid group of 3-string braids and if $p\in S_{f,n}$, then 
$\pi_1^{loc}(B_f,p)=\mathbb Z\times\mathbb Z$ is the free abelian group generated by $\gamma_1$ and $\gamma_2$.

Denote by $\iota_*:\pi_1^{loc}(B_f,p)=\pi_1(V_{p,\varepsilon}\setminus B_f,b)\to \pi_1(\mathbb P^2\setminus B_f,b)$ a homomorphism defined  
by  imbedding $\iota: V_{p,\varepsilon}\hookrightarrow \mathbb P^k$. The {\it local monodromy group} $G_{f,p}:=\text{im}\, f_{\mid W_p*}:=\text{im}\, f_*\circ \iota_*\subset G_f=\mathbb S_d$ of $f$ at a point $p\in S_{f,2}$ is defined by the cover $f_{\mid W_p}:W_p=f^{-1}(V_{p,\varepsilon})\to V_{p,\varepsilon}$. Therefore $G_{f,p}$ is a nontrivial  subgroup of $\mathbb S_d$, since $f$ is branched at $p$. 

The group $G_{f,p}$ is generated by two transpositions $f_{*}(\gamma_1)$ and $f_*(\gamma_2)$. 
Therefore $G_{f,p}=\mathbb S_3$ if $p\in S_{f,c}$ and  $G_{f,p}=\mathbb S_2\times\mathbb S_2$ if $p\in S_{f,n}$, since $X$ is non-singular at the points  of $f^{-1}(p)$. Hence, $W_p=\sqcup_{i=1}^{d-2}W_{p,i}\subset X$ is a disjoint union of $d-2$ non-singular connected neighborhoods such that 
if $p\in S_{f,c}$ then $\deg f_{\mid W_{p,1}}=3$ and $\deg f_{\mid W_{p,i}}=1$ (i.e., $f_{\mid W_{p,i}}:W_{p,i}\to  V_{p,\varepsilon}$ are biholomorphic mappings) for $i \geq 2$, and if $p\in S_{f,n}$ then $\deg f_{\mid W_{p,i}}=2$ for $i\leq 2$ and $\deg f_{\mid W_{p,i}}=1$ for $ i \geq 3$.
In addition, since $V_{p,\varepsilon}=\mathcal B_{\varepsilon}\times\mathcal D_{\varepsilon}^{k-2}$ and $B_f\cap V_{p,\varepsilon}=B_{\varepsilon}\times \mathcal D^{k-2}_{\varepsilon}$, then $W_p=W'_p\times \mathcal D^{k-2}_{\varepsilon}$ and $f_{\mid W_p}=f_{\mid \mathcal W'_p}\times\text{id}$, where 
$W_p'=f_p^{-1}(\mathcal B_{\varepsilon})$ (here and below, $\mathcal B_{\varepsilon}$ is the fiber over $b^{\prime\prime}\in \mathcal D_{\varepsilon}^{k-2}$ of the projection  $\mathcal B_{\varepsilon}\times\mathcal D_{\varepsilon}^{k-2}\to \mathcal D_{\varepsilon}^{k-2}$)  and $f_{\mid \mathcal W_p'}$ is the restriction of the cover $f$ to $W_p'$ that is defined by the monodromy homomorphism 
$f_{\mid W_p*}:\pi_1(\mathcal B_{\varepsilon}\setminus B_{\varepsilon},b')=\pi_1(V_p\setminus B_f,b)\to\mathbb S_d$,  
and $\text{id}:\mathcal D^{k-2}_{\varepsilon}\to\mathcal D^{k-2}_{\varepsilon}$ is the identical mapping. Denote $W_{p,i}'=W_p'\cap W_{p,i}$. We have 
$\deg f_{\mid W_{p,i}'}=\deg f_{\mid W_{p,i}}$. 

The following Lemma directly follows from obvious Lemma \ref{lem9}.
\begin{lem} \label{lem2}
If $p\in S_{f,n}$, then for $i\in\{ 1,2\}$ there are local coordinates $w_1,w_2$ in $W_{p,i}'$ such that $f_{\mid W_{p,i}'}: W'_{p,i}\to\mathcal B_{\varepsilon}$ is given by functions  $z_i=w_i$ and $z_{\overline i}=w_0^2$, where $\overline i\in \{ 1,2\} \setminus\{ i\}$. 
\end{lem}
 
\begin{lem}\label{lem9} Let $g: Y\to\mathcal D_{\varepsilon}^k$, $k>2$, be a two-sheeted cover  with the discriminant $B_g\subset \mathcal D_{\varepsilon}^k$, where $Y$ is a non-singular variety and let $Z\subset \mathcal D_{\varepsilon}^k$ be a non-singular variety, $\dim Z=k-1$. Then $g^{-1}(Z)$ is non-singular if and only if $Z$ and $B_{g}$ meet transversally at each point $q\in B_g\cap Z$.
\end{lem}

A proof of the  following Theorem can be found in \cite{K_G}. 
\begin{thm}\label{T} Let $(W,o')$ and $(V,o)$ be connected germs of non-singular 
surfaces and $f: (W,o')\to (V,o)$ a finite holomorphic map, $\deg f=3$. Then 
\begin{itemize}  \item[$(1)$]  there are  local coordinates $w_1,w_2$ in $(W,o')$ and $z_1,z_2$ in $(V,o)$, and a non-negative integer $n\in \mathbb Z_{\geq 0}$ such that $f$ is one of the covers $f_n:(W,o')\to (V,o)$ given by
$$\begin{array}{l}
z_1=w_1^3-3nw_2^nw_1, \qquad n\in \mathbb Z_{\geq 0}, \\
z_2=w_2; \end{array} $$ 
 \item[$(2)$]  if  $n\geq 1$, then the curve $(R_{f_n},o')$ is given by equation $w_1^2-nw_2^n=0$, the curve $(C_{f_n},o')$ is given by equation $w_1^2-4nw_2^n=0$, and the curve $(B_{f_n},o)$ is given by equation $z_1^2- 4nz_2^{3n}=0$; 
\item[$(3)$] the curves $R_{f_n}$ and $C_{f_n}$ are nonsingular at the point $o'$  if and only if $n=1$ and in this case  the intersection number $(R_{f_1},C_{f_1})_p=2$. \end{itemize}
\end{thm}

The following Lemma directly follows from Theorem \ref{T}.
\begin{lem} \label{lem1}
If $p\in S_{f,c}$, then there are local coordinates $w_1,w_2$ in $W_{p,1}'$ such that $f_{\mid W_{p,1}'}: W'_{p,1}\to\mathcal B_{\varepsilon}$ is given by functions  $z_1=-\frac{1}{2}(w_1^3-3w_2w_1)$ and $z_2=w_2$.
\end{lem}

Consider the intersection $V_{p,\varepsilon}\cap S_{f,2}$. If $p\in S_{f,c}$ then $V_{p,\varepsilon}\cap S_{f,2}=o\times \mathcal D_{\varepsilon}$, where 
$o=(0,0)\in \mathcal B_{\varepsilon}$, and $f^{-1}(V_{p,\varepsilon}\cap S_{f,c})\cap W_{p}=R_{c,1}\sqcup(\sqcup_{i=2}^{d-2}C_{c,i})$, where $R_{c,1}\subset W_{p,1}\cap R_f$ and  $C_{c,i}\subset W_{p,i}\cap C_f$. It is easy to see that $f_{\mid\, R_{c,1}}:R_{c,1}\to V_{p,\varepsilon}\cap S_{f,c}$ is one-to-one map. Respectively,  $f^{-1}(V_{p,\varepsilon}\cap S_{f,n})\cap W_{p}=(R_{n,1}\sqcup R_{n,2})\sqcup(\sqcup_{i=3}^{d-2}C_{n,i})$, where $R_{n,i}\subset W_{p,i}\cap R_f$ for $i=1,2$ and  $C_{n,i}\subset W_{p,i}\cap C_f$ for $i\geq 3$. It is easy to see that 
$f_{\mid\, R_{n,i}}:R_{n,i}\to V_{p,\varepsilon}\cap S_{f,n}$ is one-to-one map for $i=1,2$. Cosequently, we have the following 
\begin{claim} \label{cl} The intersections $R_{f,c}=f^{-1}(S_{f,c})\cap R_f$ and  $R_{f,n}=f^{-1}(S_{f,n})\cap R_f$ are non-singular quasi-projective varieties, $f_{\mid\, R_{f,c}}:R_{f,c}\to S_{f,c}$ is an isomorphism and $f_{\mid\, R_{f,n}}:R_{f,n}\to S_{f,n}$ is a non-ramified morphism, 
$\deg f_{\mid\, R_{f,n}}=2$.
\end{claim}

\begin{lem}\label{lem3} Let $W_{\varepsilon}\subset X$ be a connected neighborhood in a nonsingular variety $X$ and $f:W_{\varepsilon}\to \mathcal D^k_{\varepsilon}=\{ (z_,\dots, z_k)\in \mathbb C^k\, \mid\, |z_i|< \varepsilon\}$ be a finite holomorphic map. 
Assume that $W_{\varepsilon,\Pi_k}=f^{-1}(\Pi_k)\subset W_{\varepsilon}$ is non-singular, where $\Pi_k=\{ z_k=0\}\subset \mathcal D^k_{\varepsilon}$, and assume also that there are local coordinates $w_1,\dots,w_{k-1}$ in $W_{\varepsilon,\Pi_k}$ such that $f_{\mid W_{\varepsilon,\Pi_k}}: W_{\varepsilon,\Pi_k}\to \Pi_k$ is given by functions  $z_1=-\frac{1}{2}(w_1^3-3w_1w_2)$ and $z_i=w_i$ for $i=2,\dots, k-1$. 

Then there are a positive $\varepsilon_1<\varepsilon$ and  a neighborhood $V_{\varepsilon_1}\subset \mathcal D^k_{\varepsilon}$, 
$$V_{\varepsilon_1}\simeq \mathcal D_{\varepsilon_1}^k=\{ (\widetilde z_1,\dots, \widetilde z_k)\in\mathbb C^k\,\mid\, |\widetilde z_i|<\varepsilon_1\},$$  and there are coordinates $\widetilde w_1,\dots, \widetilde w_k$ in $W_{\varepsilon_1}=f^{-1}(V_{\varepsilon_1})$ such that  
$f_{\mid W_{\varepsilon_1}}:W_{\varepsilon_1}\to \mathcal D^k_{\varepsilon_1}$ is given by functions  $\widetilde z_1=-\frac{1}{2}(\widetilde w_1^3-3\widetilde w_1\widetilde w_2)$,  $\widetilde z_i=\widetilde w_i$ for $i=2,\dots,k$.   
\end{lem}
\proof The inverse image $p=f^{-1}(o)$ of the origin $o=(0,\dots, 0)\in\mathcal D^k_{\varepsilon}$ of the $k$-disc $\mathcal D^k_{\varepsilon}$ is a single point, since $f_{\mid W_{\varepsilon,\Pi_k}}^{-1}(o)$ is a single point. Consequently, degree $\deg_p f$ of the holomorphic map $f$ at the point $p$ is equal to $\deg_p f=\deg_p f_{\mid W_{\varepsilon,\Pi_k}}=3$ and therefore the local ring $\mathcal O^{\it an}_{W_{\varepsilon},p}$  of holomorphic functions  at the point $p\in W_{\varepsilon}$ is a module of finite type over the local  ring  $\mathcal O^{\it an}_{\mathcal D_{\varepsilon}^k,o}$. Let us identify the ring $\mathcal O^{\it an}_{\mathcal D_{\varepsilon}^k,o}$ with its image $f^*(\mathcal O^{\it an}_{\mathcal D_{\varepsilon}^k,o})\subset 
\mathcal O^{\it an}_{W_{\varepsilon},p}$ and identify $\mathcal O^{\it an}_{\Pi_k,o}$ with its image 
$f_{W_{\varepsilon,\Pi_k}}^*(\mathcal O_{\Pi_k,o})\subset \mathcal O^{\it an}_{W_{\varepsilon,\Pi_k},p}$. Then $\mathcal O^{\it an}_{W_{\varepsilon,\Pi_k},p}=\mathcal O^{\it an}_{\Pi_k,o}w_1^2\oplus \mathcal O^{\it an}_{\Pi_k,o}w_1\oplus\mathcal O^{\it an}_{\Pi_k,o}$ and 
\begin{equation} \label{w1} w_1^3=3z_2w_1-2z_1.\end{equation}

The imbedding $\iota:W_{\varepsilon,\Pi_k} \hookrightarrow W_{\varepsilon}$ defines a surjective homomorphism 
$$\iota^*:\mathcal O^{\it an}_{W_{\varepsilon},p}\to \mathcal O^{\it an}_{W_{\varepsilon,\Pi_k},p}, \qquad \iota^*(g)= g_{\mid W_{\varepsilon,\Pi_k}}\,\, \text{for}\,\, g\in \mathcal O^{\it an}_{W_{\varepsilon},p}.$$
Let us choose a function in $\iota^{-1}(w_1)\subset\mathcal O^{\it an}_{W_{\varepsilon},p}$ and denote it by the same letter $w_1$. Then the functions $w_1$ and $z_2, \dots,z_k$ are local parameters in $W_{\varepsilon}$ at $p$ and 
\begin{equation}\label{w2} w_1^3+6g_2(z_1,\dots z_k)w_1^2-3g_1(z_1,\dots,z_k)w_1+2g_0(z_1,\dots,z_k)=0,\end{equation}
where $g_i$ are functions in $\mathcal O^{\it an}_{\mathcal D^k_{\varepsilon},o}$ such that
$$\begin{array}{lcl} g_2(z_1,\dots,z_k) & = & z_k\widetilde g_2(z_1,\dots,z_k), \\ 
g_1(z_1,\dots,z_k) & = & z_2+z_k\widetilde g_1(z_1,\dots,z_k), \\ g_0(z_1,\dots,z_k) & = & z_1+z_k\widetilde g_0(z_1,\dots,z_k), \end{array}$$
since the restriction of equality (\ref{w2})  to $W_{\varepsilon,\Pi_k}$ is equality (\ref{w1}). 
It is easy to see that the functions 
$$\begin{array}{lcll}\widetilde z_1  & = & z_1+z_k(\widetilde g_0+3z_2\widetilde g_2+3z_k\widetilde g_1\widetilde g_2+8z_k^2\widetilde g_2^3), & \\ 
\widetilde z_2 & = & z_2+z_k(\widetilde g_1+4z_k\widetilde g_2^2),  & \\ \widetilde z_i & = & z_i &  \text{for}\,\, i=3,\dots, k \end{array}$$ 
are coordinates in $V_{\varepsilon_1}\simeq \mathcal D^k_{\varepsilon_1}=\{ (\widetilde z_1,\dots, \widetilde z_k)\in\mathbb C^k\,\mid\, |\widetilde z_i|<\varepsilon_1\}$ at the point $o$ if $\varepsilon_1$ is a sufficiently small positive number and  
the functions 
$$\begin{array}{lcll}\widetilde w_1  & = & w_1+2z_k\widetilde g_2, & \\ 
\widetilde w_2 & = & w_2+z_k(\widetilde g_1+4z_k\widetilde g_2^2),  & \\ \widetilde w_i & = & w_i &  \text{for}\,\, i=3,\dots, k \end{array}$$ 
are coordinates in $W_{\varepsilon_1}=f^{-1}(V_{\varepsilon_1})$. 

By doing the Tschirnhausen transformation in equality (\ref{w2}), it is easy to check that in coordinates  $\widetilde z_1,\dots, \widetilde z_k$ in $V_{\varepsilon_1}\subset\mathcal D_{\varepsilon}^{k}$ and in coordinates $\widetilde w_1,\dots, \widetilde w_k$ in $W_{\varepsilon_1}$, the map $f_{\mid W_{\varepsilon_1}}:W_{\varepsilon_1}\to V_{\varepsilon_1}$ is given by functions  
$$\widetilde z_1=-\frac{1}{2}(\widetilde w_1^3-3\widetilde w_1\widetilde w_2)\,\, \text{and}\,\, \widetilde z_i=\widetilde w_i\,\, \text{for}\,\, i=2,\dots,k.$$   

\begin{df} Let $D\subset X$ be a divisor in a non-singular variety $X$, $\dim X=k>1$. A  point $p\in D$ is said to be a singular $(k-2)$-dimensional point of $D$ of singularity type $A_{m-1}$ if there are a neighborhood (in complex topology) $U_p\subset X$ of the point $p$ and local coordinates $w_q,\dots, w_k$ in $U_p$ such that $D\cap U_p$ is given by equation $w_1^2-w_2^m=0$. 
\end{df} 

\begin{lem}\label{lem3'} Let $W_{\varepsilon}\subset X$ be a connected neighborhood in a nonsingular variety $X$ and $f:W_{\varepsilon}\to \mathcal D^k_{\varepsilon}=\{ (z_,\dots, z_k)\in \mathbb C^k\, \mid\, |z_i|< \varepsilon\}$, $k\geq 3$, be a finite holomorphic map such that $\deg _{o'} f=3$ and $o'=f^{-1}(o)$ is a $(k-2)$-dimensional singular point of the divisor $R_f\cup C_f$ of singularity type $A_3$, where $o=(0,\dots,0)$. 
Assume that $W_{\varepsilon,\Pi_k}=f^{-1}(\Pi_k)\subset W_{\varepsilon}$ is non-singular and meets transversally $R_f\cup C_f$ at the point $o'$, where  
$\Pi_k=\{ z_k=0\}\subset \mathcal D^k_{\varepsilon}$.

Then there are positive $\varepsilon_1<\varepsilon$ and  a neighborhood $V_{\varepsilon_1}\subset \mathcal D^k_{\varepsilon}$, 
$$V_{\varepsilon_1}\simeq \mathcal D^k_{\varepsilon_1}=\{ (\widetilde z_1,\dots, \widetilde z_k)\in\mathbb C^k\,\mid\, |\widetilde z_i|<\varepsilon_1\},$$  and there are coordinates $\widetilde w_1,\dots, \widetilde w_k$ in $W_{\varepsilon_1}=f^{-1}(V_{\varepsilon_1})$ such that  
$f_{\mid W_{\varepsilon_1}}:W_{\varepsilon_1}\to \mathcal D^k_{\varepsilon_1}$ is given by functions  $\widetilde z_1=-\frac{1}{2}(\widetilde w_1^3-3\widetilde w_1\widetilde w_2)$, $\widetilde z_i=\widetilde w_i$ for $i=2,\dots,k$, and $B_{f_{\mid\, W_{\varepsilon_1}}}$ is given by equation 
$\widetilde z_1^2-\widetilde z_2^3=0$. 
\end{lem}
\proof Applying induction on $k$, Lemma \ref{lem3'} directly follows from Theorem \ref{T} and Lemma \ref{lem3}. \qed \\

\begin{lem} \label{A_3} Let $X\subset \mathbb P^N$ be a non-singular irreducible $k$-dimensional projective variety, $R$ and $C$ be reduced divisors in $X$, and $S$ an irreducible component of $R\cap C$. Assume that $p\in S$ is a $(k-2)$-dimensional singular point of the divisor $R\cup C$ of singularity type $A_3$. Then there is an everywhere dense Zariski open subset $S_1$ of $S$ the points of which are  $(k-2)$-dimensional singular points of the divisor $R\cup C$ of singularity type $A_3$.
\end{lem}   
\proof Since $p$ is a $(k-2)$-dimensional singular point of the divisor $R\cup C$ of singularity type $A_3$, then there is a neighborhood 
$W_{\varepsilon}\subset X$ of the point $p$, 
$$W_{\varepsilon}\simeq \mathcal D_{\varepsilon}^k=\{ (w_1,\dots, w_k)\in\mathbb C^k\,\mid\, |w_i|<\varepsilon\},$$ 
such that $p=(0,\dots,0)$, $R\cap W_{\varepsilon}$ is given by equation $ w_1-w_2^2=0$, and $C\cap W_{\varepsilon}$ is given by equation $w_1+w_2^2=0$. Note that $R$, $C$ and $S$ are nonsingular at $p$. In a new coordinate system $\widetilde z_1=w_1-w_2^2, \widetilde z_i=w_i$, $i=2,\dots, k$, $R\cap W_{\varepsilon}$ is given by equation $\widetilde z_1=0$ and $C\cap W_{\varepsilon}$ is given by equation $\widetilde z_1+2\widetilde z_2^2=0$.

Using standard statements in algebraic geometry contained, for example, in \cite{Sh}, it is easy to show u that there is a linear projection $\text{pr} :\mathbb P^N\to \mathbb P^k$ such that $\text{pr}_{\mid\, X}$ is a finite morphism and there are a Zariski open subset $U$ of $X$ and affine space $\mathbb C^N=\text{Spec}\, \mathbb C[z_1,\dots, z_N]\subset \mathbb P^N$ such that 
\begin{itemize} \item[$(1)$] $p\in W_{\varepsilon}\subset U\subset \mathbb C^N$,
\item[$(2)$] $\text{pr}_{\mid\, \mathbb C^N}: (z_1,\dots, z_N)\mapsto (z_1,\dots, z_k)$,
\item[$(3)$] $R\cap U$, $C\cap U$, $S\cap U$ are non-singular and and irreducible, 
 \item[$(4)$] $R\cap U=\{h_1=0\}$, $S\cap U=\{ h_1=h_2=0\}$, where  $h_1,\, h_2\in\mathbb C[U]$ are regular functions in $U$, 
\item[$(5)$] $h_1,h_2, z_3-z_3(q),\dots,z_k-z_k(q)$ are local parameters in $U$ at each point  $q\in S\cap U$. 
\end{itemize}

The restriction $C_{\mid\, R\cap W_{\varepsilon}}$ of the divisor $C$ to $R\cap W_{\varepsilon}$ is $2(S\cap W_{\varepsilon})\in \text{Div}(R\cap W_{\varepsilon})$. Therefore the restriction $C\cap (R\cap U)$ of $C$  to $R\cap U$ is $2S\cap (R\cap U)\in\text{Div}(R\cap U)$ and hence, $C$ is given in $U$ by equation $g_2h_2^2+h_1 g_1=0$, where, without loss of generality. we can assume that $g_1,g_2\in \mathbb C[U]$. The function $g_1$ does not vanish at any point $q\in S\cap U$, since $C$ is non-singular at each $q\in S\cap U$ and $g_2$ does not vanish identically in $U$. In addition, we can assume that the function $g_2$ does not vanish at any point $q\in S\cap U$, since we can replace $U$ with $U'=U\setminus \{ g_2=0\})$. 
Therefore for each point $q\in S\cap U'$ the holomorphic functions $w_1=g_1h_1$, $w_2=\sqrt{g_2}h_2$ and $w_i=z_i-z_i(q)$, $i=3,\dots, k$, are local coordinates in some open (in complex topology) neighborhood $W_q$ of $q$ and in this neighborhood the divisor $(R\cup C)\cap W_q$ is given by equation $w_1(w_1+w_2^2)=0$, i.e.,  each point $q\in S\cap U'$ is a $(k-2)$-dimensional singular point of the divisor $R\cup C$ of singularity type $A_3$. \qed \\

\section{Inductive definition of covers, generic outside of codimension three} \label{ind-def}
\subsection{Auxiliary statements.} Before to give another definition of covers, generic outside of codimension three, we prove several lemmas, that are needed in order to prove the equivalence of definitions.  

\begin{lem}\label{dim2} Let $S\subset \mathbb P^k$ be a non-singular quasi-projective variety. $\dim S=k-2\geq 1$. Then  the set $\mathcal G_{S}$  that consists of the hyperplanes $\Pi$ meeting $S$ transversally at each point of $\Pi \cap S$, is an everywhere dense Zariski open subset  of $\hat{\mathbb P}^k$.   
\end{lem}
\proof Denote $\mathcal T =\{ (p,\Pi)\in S\times \hat{\mathbb P}^k\, \mid\, \overline T_{S,p}\subset \Pi\},$  where $\overline T_{S,p}\subset \mathbb P^k$ is the closure in $\mathbb P^k$ of the tangent space  $T_{S,p}$ of $S$ at $p\in S$, and let $\psi_1:\mathcal T\to S$ and $\psi_2:\mathcal T\to\hat{\mathbb P} ^k$ be the restrictions to $\mathcal T$ of the projections $\text{pr}_1:S\times \hat{\mathbb P}^k\to S$ and $\text{pr}_2:S\times \hat{\mathbb P}^k\to\hat{\mathbb P} ^k$. 

We have $\dim \psi^{-1}_1(p)=1$, since $\dim T_{S,p} =k-2$. Therefore $\dim \mathcal T=k-1$ and hence, $\dim \overline{\psi_2(\mathcal T)}\leq k-1$. Therefore $\mathcal G_{S}=\hat{\mathbb P}^k\setminus \overline{\psi_2(\mathcal T)}$ is an everywhere dense Zariski open subset  of $\hat{\mathbb P}^k$.    \qed 
 
\begin{lem}\label{dim1} A set of hyperplanes in $\mathbb P^k$  that does not contain a subvariety $S$ of $\mathbb P^k$ is an everywhere dense Zariski open subset of $\hat{\mathbb P}^k$.
\end{lem} 
\proof   The set  of hyperplanes in $\mathbb P^k$ which contain $S$ is a linear subspace of $\hat{\mathbb P}^k$ of dimension $k-m-1$, where $m$ is the minimum of dimensions of linear subvarieties $L$ of $\mathbb P^k$ containing $S$. \qed

\begin{lem} \label{proj1} Let $\overline S\subset \mathbb P^k$ be a projective variety, $\dim \overline S=k-2$. Then there exists an every dense Zariski  open  subset $G_{\overline S}\subset \text{Gr}(k,1)$ of lines in $\mathbb P^k$ that do not meet $\overline S$. 
\end{lem}
\proof Denote $G=\{ (p,L)\in S\times \text{Gr}(k,1)\, \mid \, p\in L\cap S\}$ and let $\varphi_1:G\to S$, $\varphi_2:G \to \text{Gr}(k,1)$ be the restrictions to $G$ of the projections $\text{pr}_1:S\times \text{Gr}(k,1)\to S$, $\text{pr}_2:S\times \text{Gr}(k,1) \to \text{Gr}(k,1)$. Then 
$\dim \varphi_1^{-1}(p)=k-1$ for each point $p\in S$ and hence, $\dim G=2k-3<\dim \text{Gr}(k,1)$. Therefore $G_S=\text{Gr}(k,1)\setminus \varphi_2(G)$ is a desired open subset of $\text{Gr}(k,1)$. \qed
 
\begin{lem}\label{proj2} Let $G_1\subset \text{Gr}(k,1)$ be an everywhere dense Zariski open subset of lines in $\mathbb P^k$ and $G_2$ an everywhere dense Zariski open subset of hyperplanes in $\mathbb P^k$. Then there exist a line $L\in G_1$ and hyperplanes $\Pi_1,\dots, \Pi_{k-1}\in G_2$ such that $\displaystyle L=\bigcap_{i=1}^{k-1}\Pi_i$.
\end{lem} 
\proof Obviously, the product $G_2^{k-1}$ of $k-1$ copies of $G_2$ is an everywhere dense Zariski open subset of the product $\hat{\mathbb P}^{k^{k-1}}$ of $k-1$ copies of the projective space  $\hat{\mathbb P}^{k}$. 

Let $a_{i,1}x_1+\dots +a_{i,k+1}x_{i,k+1}=0$, $i=1,\dots,k-1$, be an equation of a hyperplane $\Pi_i$ in $\mathbb P^k$. Denote by 
$ r(\Pi_1,\dots, \Pi_{k-1})$ rank of the matrix   
$$ A=\left( \begin{array}{ccc} a_{1,1}, & \cdots &  a_{1,k+1} \\ 
\vdots & \vdots & \vdots \\ 
a_{k-1,1} & \cdots   & a_{k-1,k+1} \end{array} \right) $$ 
and $C_2=\{ (\Pi_1,\dots,\Pi_{k-1})\in \hat{\mathbb P}^{k^{k-1}}=\hat{\mathbb P}^k\times\dots\times \hat{\mathbb P}^k\, \mid\, r(\Pi_1,\dots,\Pi_{k-1})<k-1\}$. Consider a morphism $\lambda:\hat{\mathbb P}^{k^{k-1}}\setminus C_2\to \text{Gr}(k,1)$ sending $(\Pi_1,\dots, \Pi_{k-1})$ to $L=\cap_{i=1}^{k-1}\Pi_i$. Obviously, $\lambda$ is a surjection.  
Therefore $\lambda^{-1}(G_1)$ is an everywhere dense Zariski open subset of $\hat{\mathbb P}^{k^{k-1}}$ and hence, $G_2^{k-1}\cap \lambda^{-1}(G_1)$ is an everywhere dense Zariski open subset of $\hat{\mathbb P}^{k^{k-1}}$. In particular, $G_2^{k-1}\cap \lambda^{-1}(G_1)\neq \emptyset$. Therefore a point $(\Pi_1,\dots,\Pi_{k-1})\in G_2^{k-1}\cap \lambda^{-1}(G_1)$ and $L=\lambda(\Pi_1,\dots,\Pi_{k-1})$ are desired hyperplanes and a line in $\mathbb P^k$. \qed 

\begin{lem} \label{proj3} Let $\overline S\subset \mathbb P^k$, $L\in G_{\overline S}=G_1\subset \text{Gr}(k,1)$ and $\Pi_1,\dots, \Pi_{k-1}\in G_2\subset \hat{\mathbb P}^k$  be the same as in Lemmas  \ref{proj1} and \ref{proj2} and let  $S$ be an everywhere dense Zariski open subvariety of $\overline S$. Then there is  a quasi-projective variety $S_{L}\subset S$ that is closed in $S$, $\dim S_{L}\leq k-3$, and such that 
for each point $p\in \mathcal S=S\setminus S_{L}$ there is a hyperplane $\Pi\in G_2$ meeting $\mathcal S$ at the point $p$. 
\end{lem} 
\proof  Let a finite morphism $\phi :\overline S\to  \mathbb P^{k-2}$ be the restriction to $\overline S$ of the linear projection $\text{pr}_L:\mathbb P^k\to \mathbb P^{k-2}$ with center in $L$ and $B_{\phi}$ be  its discriminant,  $\dim \psi^{-1}(B_{\psi})=k-3$.  

Denote  $\mathcal P_L$ the linear subspace of $\hat{\mathbb P}^k$ containing the set $\{ \Pi_1,\dots, \Pi_{k-1}\}$. We have  
$\mathcal P_L\simeq\hat{\mathbb P}^{k-2}$. For a point $p\in \mathbb P^{k-2}$ denote $\mathcal P_p=\{ \Pi\in \mathcal P_L 
\mid p\in \text{pr}_L(\Pi)\}$. Then  $\mathcal P_p$ is a hyperplane in $\mathcal P_L$. 

The set $\mathcal P_L\cap G_2$ is an evrywhere dense Zariski open subset of $\mathcal P_L$. Therefore $\widetilde G_2=\{ \widetilde{\Pi}\in \hat{\mathbb P}^{k-2}\, \mid\, \text{pr}_L^{-1}(\widetilde{\Pi})\in \mathcal P_L\cap G_2\}$ is an evrywhere dense Zariski open subset of $\hat{\mathbb  P}^{k-2}\simeq \mathcal P_L$.
Hence, $N=\{ p\in \mathbb P^{k-2}\, \mid\, \mathcal P_p\cap (\mathcal P_L\cap G_2)=\emptyset \}$ is a finite set of points (may be empty). Then $S_{L}=S\cap\psi^{-1}(B_{\phi}\cup N)$ is a desired set. \qed 

\subsection{$Ind$-generic covers of the projective spaces.} \label{ind-cov} Note that if $f: X \to \mathbb P^2$ is a generic cover of the projective plane, then there exists an everywhere dense open subset $\mathcal G_f$ of the set $\hat{\mathbb P}^2$ of lines in $\mathbb P^2$, namely $\mathcal G_f=\hat{\mathbb P}^2\setminus \hat B_f$, such that $f_{|X_{\Pi}}:X_{\Pi}\to\Pi=\mathbb P^1$ is a generic cover of $\mathbb P^1$  for all lines $\Pi\in\mathcal G_f$, where $X_{\Pi}=f^{-1}(\Pi)$ and $\hat B_f\subset \hat{\mathbb P}^2$ is the dual curve of the discriminant $B_f$. 
We use this property to define the next subset of the set of finite morphisms from nonsingular projective varieties to projective spaces.

Let $X$ be a nonsingular projective variety, $\dim X=k$. A finite morphism $f: X\to\mathbb P^k$ is said to be {\it$ind$-generic cover of $\mathbb P^k$,} if 
\begin{itemize} \item[($\text{\bf i}$)] in the case of $k\leq 2$, the morphism $f$ is a generic cover, 
\item[($\text{\bf ii}$)] in the case of  $k\geq 3$, there exist  an everywhere dense Zariski open subset $\mathcal S_{f}$ of $\text{Sing}\, B_f$ containing only nonsingular points of $\text{Sing}\, B_f$ and an everywhere dense Zariski open subset $\mathcal G_f\subset Gr_{k,k-1}=\hat{\mathbb P}^k$ such that 
    \begin{itemize} \item[$(\text{\bf ii})_1$] $f_{|X_{\Pi}}:X_{\Pi}=f^{-1}(\Pi)\to \Pi=\mathbb P^{k-1}$ are $ind$-generic covers for all hyperplanes  $\Pi\in\mathcal G_f\subset\hat{\mathbb P}^k$, 
    \item[$(\text{\bf ii})_2$] $\mathcal S_{f_{\mid X_{\Pi}}}=\mathcal S_f\cap X_{\Pi}$ for each hyperplane $\Pi\in \mathcal G_f$ and 
    for each point $p\in\mathcal S_f$ there exists a hyperplane $\Pi\in \mathcal G_f$ meeting transversally $\mathcal S_f$ at the point $p$ 
        (if $k=3$ then $\mathcal S_{f_{\mid X_{\Pi}}}=\text{Sing}\, B_{f_{\mid X_{\Pi}}}$).
    \end{itemize}
\end{itemize}

\begin{lem}\label{irred} The discriminant $B_f$ of an $ind$-generic cover $f:X\to\mathbb P^k$ is an irreducible variety if $k\geq 2$. \end{lem}
\proof For $k=2$, the discriminant  $B_f$ of an $ind$-generic cover $f:X\to\mathbb P^3$ is an irreducible variety by definition of generic covers of $\mathbb P^2$. 

For $k\geq 3$, assume that the discriminant $B_f$ of an $ind$-generic cover $f:X\to\mathbb P^k$ is not irreducible, but for all $ind$-generic covers of $\mathbb P^{k-1}$ their discriminants are irreducible varieties. Let $B_1$ and $B_2$ be two irreducible components of $B_f$. 

By definition of $ind$-generic covers, there is  an everywhere  dense set $\mathcal G_f\subset Gr_{k,k-1}=\hat{\mathbb P}^k$ such that $f_{|X_{\Pi}}:X_{\Pi}\to \Pi=\mathbb P^{k-1}$ are $ind$-generic covers for all $\Pi\in\mathcal G_f\subset\hat{\mathbb P}^k$. Therefore the discriminants $B_{f_{\mid X_{\Pi}}}=B_f\cap \Pi$ are irreducible varieties for all $\Pi\in\mathcal G_f$. 

On the other  hand, since $\mathcal G_f$ is  an everywhere  dense subset of  $\hat{\mathbb P}^k$, then, by Lemma  \ref{dim1},  there is $\Pi\in\mathcal G_f$ such that $\dim B_{f}\cap\Pi=k-2$ and  each irreducible component of $B_1\cap B_2$ is not contained in $\Pi$. Therefore $B_{f_{\mid X_{\Pi}}}$ can not be an irreducible variety, since   $B_1\cap \Pi\neq B_2\cap\Pi$ and $B_i\cap \Pi$, $i=1,2$, are hypersurfaces in $\Pi$. \qed

\begin{lem}\label{coef} Let $f:X\to\mathbb P^k$  be an $ind$-generic cover. Then $f^*(B_f)=2R_f+C_f$ and $f_{\mid R_f}:R_f\to B_f$ is a birational morphism.
\end{lem}
\proof Let  $f^*(B_f)= \sum r_iR_{f,i}+C_f$, where $R_{f,i}$ are irreducible components of the ramification divisor of $f$. We have 
\begin{equation}\label{eq1}\deg f=\sum r_i\deg f_{\mid R_{f,i}}+\deg f_{\mid C_f}.
\end{equation}

Using induction on $k$, we can assume that Lemma \ref{coef} is true for  ind-generic covers of the projective hyperplane $\mathbb P^{k-1}$ and, in particular, it is true for  $f_{\mid X_{\Pi}}: X_{\Pi}\to\Pi=\mathbb P^{k-1}$, where $\Pi\in\mathcal G_f$. Therefore $f^*_{\mid X_{\Pi}}(B_{f_{\mid X_{\Pi}}})= 2R_{f_{\mid X_{\Pi}}}+C_{f_{\mid X_{\Pi}}}$ and $f_{\mid X_{\Pi}}:R_{f_{\mid X_{\Pi}}}\to B_{f_{\mid X_{\Pi}}}$ is a birational morphism. 

We have $B_{f_{\mid \Pi}}=B_f\cap \Pi$, $R_{f_{\mid X_{\Pi}}}=R_f\cap X_{\Pi}$, $C_{f_{\mid X_{\Pi}}}=C_f\cap X_{\Pi}$, and 
$$\deg f_{\mid X_{\Pi}}=2 +\deg f_{\mid C_{f_{\mid X_{\Pi}}}}.$$
Therefore it follows from (\ref{eq1}) that $f_{\mid R_f}:R_{f}\to B_f$ is a birational morphism and $f^*(B_f)=2R_f+C_f$, since $\deg f=\deg f_{\mid X_{\Pi}}$ and $\deg f_{\mid C_f}=\deg f_{\mid C_{f_{\mid X_{\Pi}}}}$.

\begin{thm} \label{thm2} A cover $f:X\to\mathbb P^k$, generic outside of codimension three, is an $ind$-generic cover.
\end{thm}   
\proof It is well-known that if $X$ is a nonsingular variety and $f:X\to\mathbb P^k$ is a finite morphism, $\deg f=2$, then $B_f$ is a  nonsingular hypersurface in $\mathbb P^k$, and obviously, the theorem is true for sets of covers $f:X\to\mathbb P^k$ in the case of $k\leq 2$. Therefore we assume that $k\geq 3$ and $\deg f\geq 3$.

We use induction on $k$. Consider a cover $f:X\to\mathbb P^k$, generic outside of codimension three and assume that the theorem is true for the sets  of finite covers of projective spaces $\mathbb P^m$ for $m<k$. Let us show that $f$ is an $ind$-generic cover.
 
By Bertini's Theorem, the set 
$$\mathcal G_{f,1}=\{ \Pi\in \hat{\mathbb P}^k\, \mid\, 
X_{\Pi}=f^{-1}(\Pi)\subset X\, \text{is a nonsingular variety}\}$$ 
is an everywhere dense Zariski open subset of $\hat{\mathbb P}^k$. Denote also by $\mathcal G_{f,2}\subset \hat{\mathbb P}^k$ the set of hyperplanes $\Pi$ in $\mathbb P^k$ such that $\Pi$ does not contain an irreducible component of $S_{f,2}\subset \text{Sing}\, B_f$ and meets transversally 
$S_{f,2}=S_{f,n}\cup S_{f,c}\subset  \text{Sing}\, B_f$ at each point of $\Pi\cap S_{f,2}$. By Lemmas \ref{dim2} and \ref{dim1},  
$\mathcal G_{f,2}$ is an everywhere dense Zariski open subset of $\hat{\mathbb P}^k$.  

Put $\mathcal G_f=\mathcal G_{f,1}\cap\mathcal G_{f,2}$ and in notation used in Lemma \ref{proj3},  we put $G_2:=\mathcal G_f$, $S:=S_{f,2}$, and 
$\mathcal S:=\mathcal S_f$. 

Let us show that $f_{\mid X_{\Pi}}:X_{\Pi}\to \Pi=\mathbb P^{k-1}$ is a cover, generic outside of codimension three if $\Pi\in \mathcal G_f$.
To prove this, first of all, we note that $B_{f_{\mid X_{\Pi}}}=B_f\cap \Pi$ and 
$$\text{Sing}\, B_{f_{\mid X_{\Pi}}}=S_{f_{\mid X_{\Pi}},3}\sqcup S_{f_{\mid X_{\Pi}},n}\sqcup S_{f_{\mid X_{\Pi}},c},$$ 
where $S_{f_{\mid X_{\Pi}},n}=S_{f,n}\cap \Pi$, $S_{f_{\mid X_{\Pi}},c}=S_{f,n}\cap \Pi$. By Lemma 2, $\text{Sing}\, B_{f_{\mid X_{\Pi}}}=
\text{Sing}\, B_{f}\cap \Pi$, since  $\Pi\in \mathcal G_{f,1}$. Therefore $S_{f_{\mid X_{\Pi}},3}=S_{f,3}\cap \Pi$ and hence, 
$\dim S_{f_{\mid X_{\Pi}},3}\leq k-4$, since $\Pi\in \mathcal G_{f,2}$.

Consider a point $p\in (S_{f,c}\cup S_{f,n})\cap \Pi$. Since $\Pi\in\mathcal G_{f,2}$ and, in notation used in Subsection \ref{loc}, $S_{f,n}\cup S_{f,c}$ is given by equations $z_1=z_2=0$ in $V_{p,\varepsilon}$, then, without loss of generality, we can assume that $\frac{\partial\psi}{\partial z_k}(p)\neq 0$, where $z_1,\dots, z_k$ are coordinates in $V_{p,\varepsilon}$ and $\psi(z_1,\dots,z_k)=0$ is an equation of $V_{p,\varepsilon}\cap\Pi$. Therefore there exists a positive $\varepsilon_1<\varepsilon$ such that $z_1,\dots, z_{k-1}, \overline z_k=\psi(z_1,\dots,z_k)$ are local coordinates in 
$$U_{p,\varepsilon_1}=\{ q\in V_{p,\varepsilon}\, \mid\, |\overline z_k|\leq \varepsilon_1\, \text{and}\, |z_i(q)|<\varepsilon_1 \, \text{for}\, i=1,\dots,k-1\}.$$ 
Hence, in the neighborhood $U_{p,\varepsilon_1}\cap \Pi$ of the point $p\in \Pi$, the functions $z_1,\dots, z_{k-1}$ are local coordinates and 
$B_{f_{\mid X_{\Pi}}}\cap U_{p,\varepsilon_1}$ is given by equation $z_1z_2=0$ if $p\in S_{f_{\mid X_{\Pi}},n}$ and  $B_{f_{\mid X_{\Pi}}}\cap U_{p,\varepsilon_1}$ is given by equation $z_1^2-z_2^3=0$ if $p\in S_{f_{\mid X_{\Pi}},c}$. 

By induction on $k$ and applying Lemma \ref{irred}, we obtain that $f_{\mid X_{\Pi}}:X_{\Pi}\to \Pi=\mathbb P^{k-1}$ is a cover, generic outside of codimension three if $\Pi\in \mathcal G_f$. Therefore, by assumption that Theorem \ref{thm2} is true for sets of finite covers of projective spaces $\mathbb P^m$ for $m\leq k-1$, the morphisms $f_{\mid X_{\Pi}}:X_{\Pi}\to \Pi=\mathbb P^{k-1}$ are $ind$-generic covers for all $\Pi\in \mathcal G_f$ if we put 
$\mathcal S_{f_{\mid X_{\Pi}}}=\mathcal S_f\cap \Pi$ and  hence, by definition of $ind$-generic covers, $f:X\to\mathbb P^k$ is an $ind$-generic cover also.  \qed 

\begin{thm} \label{thm3} An $\text{ind}$-generic cover $f:X\to\mathbb P^k$  is a cover, generic outside of codimension three with $S_{f,2}=\mathcal S_f \setminus \mathcal S_{f,3}$, where $\mathcal S_{f,3}$ is a certain closed quasi-projective subvariety of $\mathcal S_f$, $\dim \mathcal S_{f,3}\leq k-3$.
\end{thm}   
\proof As in the proof of  Theorem \ref{thm2},  we use induction on $k$ and we can assume that $k\geq 3$ and $\deg f\geq 3$.
Consider an ${ind}$-generic cover $f:X\to\mathbb P^k$  and assume that the theorem is true for the sets  of finite covers of projective spaces 
$\mathbb P^m$ for $m<k$. Let us show that $f$ is a cover, generic outside of codimension three. By Lemmas \ref{irred} and \ref{coef}, to show this, it suffices only to show that the cover $f:X\to \mathbb P^k$ has property ($\text{iii}$) in the definition of covers, generic outside of codimension three. 
 
Let $S$ be an irreducible component of $\mathcal S_f$ and let $p\in \mathcal S_f$. Then there is a hyperplane $\Pi\in \mathcal G_f$ containing the point  $p$ and therefore $S\cap \Pi$ near of $p$ is (locally) an irreducible component of $\mathcal S_{f_{\mid X_{\Pi}}}$. Hence, by assumption of induction,   $S\cap\Pi$ locally is either a subset of nodal part, or cuspidal part of the discriminant of the cover  $f_{\mid X_{\Pi}}: X_{\Pi}\to \Pi$, generic outside of codimension three.  

In notation used in Subsection \ref{loc}, there is a neighborhood $V_{p,\varepsilon}$ of the point $p$ in $\Pi$ such that  $f^{-1}_{\mid X_{\Pi}}(V_{p,\varepsilon})=W_p=\sqcup_{i=1}^{d-2}W_{p,i}\subset X_{\Pi}$ is a disjoint union of $d-2$ non-singular connected neighborhoods such that if $p\in S_{f_{\mid X_{\Pi}},c}$ then $\deg f_{\mid W_{p,1}}=3$ and $\deg f_{\mid W_{p,i}}=1$ (i.e., $f_{\mid W_{p,i}}:W_{p,i}\to  V_{p,\varepsilon}$ are biholomorphic mappings) for $i \geq 2$, and if $p\in S_{f_{\mid X_{\Pi}},n}$ then $\deg f_{\mid W_{p,i}}=2$ for $i\leq 2$ and $\deg f_{\mid W_{p,i}}=1$ for $ i \geq 3$. Obviously, reducing $V_{p,\varepsilon}$ if necessary we can find a neighborhood $\widetilde V_p\subset \mathbb P^k$ such that $\widetilde V_p\cap\Pi=V_{p,\varepsilon}$ and $f^{-1}(\widetilde V_p)=\widetilde W_p=\sqcup_{i=1}^{d-2}\widetilde W_{p,i}\subset X$ is a disjoint union of $d-2$ non-singular connected neighborhoods such that if $p\in S_{f_{\mid X_{\Pi}},c}$ then $\deg f_{\mid \widetilde W_{p,1}}=3$ and $\deg f_{\mid \widetilde W_{p,i}}=1$ for $i \geq 2$, and if $p\in S_{f_{\mid X_{\Pi}},n}$ then $\deg f_{\mid\widetilde  W_{p,i}}=2$ for $i\leq 2$ and $\deg f_{\mid \widetilde W_{p,i}}=1$ for $ i \geq 3$. 

If $S\cap\Pi$ locally is a subset of the cuspidal part of the discriminant of the cover $f_{\mid X_{\Pi}}: X_{\Pi}\to \Pi$, generic outside of codimension three, then by Lemmas \ref{lem1} and \ref{lem3}, there is a neighborhood $U_p\subset \mathbb P^k$ and local coordinates $z_1,\dots, z_k$ in it such that $B_f\cap U_p$ is given by equation $z_1^2-z_2^3=0$. 

If $S\cap\Pi$ locally is a subset of the nodal part of the discriminant of the cover $f_{\mid X_{\Pi}}: X_{\Pi}\to \Pi$, then $\deg f_{\mid\widetilde  W_{p,i}}=2$ for $i\leq 2$ and by Lemmas \ref{lem2} and \ref{lem9}, the discriminants $B_{f_{\widetilde W_i}}$ of the mappings $f_{\mid\widetilde  W_{p,i}}$ are non-singular varieties and $B_{f_{\widetilde W_1}}\cup B_{f_{\widetilde W_2}}=B_{f}\cap \widetilde V_p$.

Let $\mathbb P^k=\bigcup_{j=1}^{k+1}\mathbb C^k_j$, where $\mathbb C^k_j$ is an affine space in which the $j$-th coordinate of homogeneous coordinate system in $\mathbb P^k$ is not equal to $0$. Let $x_{j,1},\dots, x_{j,k}$ be non-homogeneous coordinates in $\mathbb C^k_j$ and $h_j(x_{j,1},\dots,x_{j,k})=0$ be an equation of $B_f\cap \mathbb C^k_j$, Denote by $H_2\subset \mathbb P^k$ a projective variety given in each affine space $\mathbb C^k_j$ by equations
$$ \text{rk}\left( \begin{array}{ccc} \frac{\partial^2 h_j}{\partial x_{j,1}\partial x_{j,1}}  & \cdots & \frac{\partial^2 h_j}{\partial x_{j,1}\partial x_{j,k}} \\ \vdots & \cdots & \vdots \\ \frac{\partial^2 h_j}{\partial x_{j,k}\partial x_{j,1}}  & \cdots & \frac{\partial^2 h_j}{\partial x_{j,k}\partial x_{j,k}} \end{array}\right) < 2.$$ 
Then $\mathcal S_{f,3}$ is the intersection of $H_2$ with the union of the all irreducible components of $\mathcal S_{f}$ that meet $S_{f_{\mid X_{\Pi}},n}$ at the points $p\in \Pi$, $\Pi \in \mathcal G_f$.  \qed 

\begin{thm} \label{gmain} Let $f:X\to \mathbb P^k$ be a finite morphism from a non-singular projective variety $X, \dim X=k\geq 3$, and let  a hyperplane $\Pi\subset \mathbb P^k$ be such that $X_{\Pi}=f^{-1}(\Pi)$ is non-singular, $\Pi$ is not an irreducible component of the discriminant $B_f$ of $f$,  
and $f_{\mid\, X_{\Pi}}:X_{\Pi}\to \Pi=\mathbb P^{k-1}$ is a cover, generic outside of codimension three. Then $f$ also is a cover, generic outside of codimension three.
\end{thm}
\proof We have $\deg f=\deg f_{\mid\, X_{\Pi}}$, since $\Pi\not\subset B_f$. Therefore for each $p\in \Pi$ and for each $q\in f^{-1}(p)$ we have $\deg_q f=\deg_q f_{\mid\, X_{\Pi}}$. In addition, we have $B_{f_{\mid X_{\Pi}}}=B_f\cap X_{\Pi}$ and $R_{f_{\mid X_{\Pi}}}=R_f\cap X_{\Pi}$.

Using arguments that used in the proof of Lemmas \ref{irred} and \ref{coef}, it is easy to show that $B_f$ is irreducible, $f^*(B_f)=2R_f+C_f$, and $f_{\mid\, R_f}:R_f \to  B_f$ is a birational morphism. Therefore, by Lemma \ref{lem9}, $\text{Sing}\,B_{f_{\mid X_{\Pi}}}=\text{Sing}\, B_f\cap X_{\Pi}$.

Let $S$ be one of irreducible components of $\text{Sing}\, B_f$. The case of $\dim S\leq k-3$ can be omitted, since we want to prove that $f$ is  a cover,  generic outside of codimension three. Therefore we can assume that $\dim S=k-2$.   Note also that $S\not\subset \Pi$, since in the opposite case, there is a point $p\in S\cap (B_{f_{\mid\, X_{\Pi}}}\setminus \text{Sing}\, B_{f_{\mid\, X_{\Pi}}})$ that contradicts Lemma \ref{lem9}.
Therefore $S\cap \Pi$ contains the closure of an irreducible component of $S_{f_{\mid X_{\Pi}},2}$. Denote this component by $T$.

If $T$ is  a component of the closure of $S_{f_{\mid X_{\Pi}},n}$, then there is a point $p\in S_{f_{_{\mid X_{\Pi}}},n}\cap S$. Therefore, by Claim \ref{cl}, $f^{-1}(p)\cap R_f$ is two points, say $q_1$ and $q_2$, and $\deg_{q_1} f=\deg_{q_2} f=2$. Hence, $p$ is a $(k-2)$-node of the discriminant  $B_f$ and therefore all points of $S\setminus (\text{Sing}\, S \cup B_{f_{\mid R_S}})$ are $(k-2)$-nodes of $R_f$, where $f_{\mid\, R_S}: R_S=R_f\cap f^{-1}(S)\to S$ is a morphism of $\deg f_{\mid\, R_S}=2$. Note that $\dim\, (\text{Sing}\, S \cup B_{f_{\mid R_S}})\leq k-3$.

If $T$ is  a component of the closure of $S_{f_{\mid X_{\Pi}},c}$, then there is a point $p\in S_{f_{\mid X_{\Pi}},c}\cap S$. Therefore the existence of Zariski open  subset $S^0$ of $S$ the points of which are  $(k-2)$-cusps of $B_f$, directly follows from Lemmas \ref{lem3} -- \ref{A_3}. \qed

\section{Proof of Theorems 1 -- 3} \label{main-proof}
\subsection{Proof of Theorem 1.}
A cover $f: X\to\mathbb P^k$, $k\geq 3$, generic outside of codimension three is an $ind$-generic cover. Therefore there is a plane $\mathbb P^2\subset \mathbb P^k$ such that $f_{\mid\, X_{\mathbb P^2}}:X_{\mathbb P^2}=f^{-1}(\mathbb P^2)\to\mathbb P^2$ is a generic cover of the plane, 
$\deg f=\deg f_{\mid\, X_{\mathbb P^2}}$, $B_{f_{\mid\, X_{\mathbb P^2}}}=B_f\cap \mathbb P^2$, and $\text{Sing}\, B_{f_{\mid\, X_{\mathbb P^2}}}= S_{f,2}\cap \mathbb P^2$ contains only ordinary node and ordinary cusps of $B_{f_{\mid\, X_{\mathbb P^2}}}$. 

If $B_{f,c}=\emptyset$, then all singular points of $B_{f_{\mid\, X_{\mathbb P^2}}}$ are only ordinary nodes. Therefore $\pi_1(\mathbb P^2\setminus  B_{f_{\mid\, X_{\mathbb P^2}}})$ is an abelian group and hence, the monodromy group $G_{}\simeq \mathbb S_{\deg f}$ is an abelian group. Cnsequently, $\deg f=2$ and $B_f$ is a nonsingular variety, since $X$ is a non-singular variety.

If $B_{f,n}=\emptyset$, then all singular points of $B_{f_{\mid\, X_{\mathbb P^2}}}$ are only ordinary cusps. Therefore by Theorem 2 in \cite{K3}, 
$\deg f=\deg f_{\mid\, X_{\mathbb P^2}}\leq 5$.

\subsection{Proof of Theorem \ref{main1}} As it was mentioned above, Theorem \ref{main1} is well-known if $k=2$.  
Therefore we will assume that $k\geq 3$ and use induction on $k$, i.e., assume that for $k\geq 3$ and for each non-singular irreducible projective variety $X_0\subset \mathbb P^{N-1}$, $\dim X_0=k-1$, there is a $(N-k -1)$-dimensional linear subspace of $\mathbb P^{N-1}$, $L\cap X_0=\emptyset$, such that the restriction $f_0:X_0\to\mathbb P^{k-1}$ of the linear projection  $\text{pr}_L:\mathbb P^{N-1}\to\mathbb P^{k-1}$ to $X_0$ is a cover, generic outside of codimension three. 

By Bertini Theorem, there is a hyperplane $\mathbb P^{N-1}\subset \mathbb P^N$ such that  $X_0=X\cap\mathbb P^{N-1}$ is a non-singular variety. 
By assamption, there is a $(N-k -1)$-dimensional linear subspace of $\mathbb P^{N-1}$, $L\cap X_0=\emptyset$, such that the restriction $f_0:X_0\to\mathbb P^{k-1}$ of the linear projection  $\text{pr}_L:\mathbb P^{N-1}\to\mathbb P^{k-1}$ to $X_0$ is a cover, generic outside of codimension three. 

Consider the restriction $f:X\to\mathbb P^{k}$ of the linear projection  $\text{pr}_L:\mathbb P^{N}\to\mathbb P^{k}$ to $X$. It is a finite morphism, since $X\cap L=X_0\cap L=\emptyset$ for $L\subset P^{N-1}$.  Note that $X_{\Pi}=X_0\not\subset R_f$, since $\deg f=\deg f_0=\dim X=d$. 

Denote $\Pi=\text{pr}_L(\mathbb P^{N-1})$, where $\mathbb P^{N-1}$ is the hyperplane containing $X_0$. Then $X_{\Pi}=f^{-1}(\Pi)=X_0$ and Theorem \ref{main1} directly follows from Theorem \ref{gmain}. \qed  

\subsection{Proof of Theorem 3.} 
Let $f_i: X\to\mathbb P^k$, $i=1,2$, be  covers, generic outside of codimension three with the same discriminant $B=B_{f_1}=B_{f_2}\subset \mathbb P^k$ and $\iota: \mathbb P^2\hookrightarrow \mathbb P^k$, $k\geq 3$, be the same plane as in the proof of Theorem 1. The existence of such $\mathbb P^2$ guarantees by Bertini Theorem, since $f_i$ are  $ind$-generic covers and the intersection of two everywhere dense Zariski open subsets in $\hat{\mathbb P}^{k-j}$, $1\leq j\leq k-2$, is not empty. 

It follows from Stein Theorem (\cite{St}) that the covers $f_1$ and $f_2$ are equivalent if and only if $d_1=d_2=d$ and there is an inner automorphism $\varphi:\mathbb S_d\to\mathbb S_d$ of $\mathbb S_d$ such that $f_{1*}=\varphi\circ f_{2*}$, where $f_{i,*}$ are the monodromy homorphisms of the covers $f_i$.

The imbedding $\iota$ defines a homomorphism $\iota_*: \pi_1(\mathbb P^2\setminus B)\to \pi_1(\mathbb P^k\setminus B)$. 
The monodromy homomorphisms $f_{i\mid\, X_{\mathbb P^2}*}:\pi_1(\mathbb P^2\setminus B)\to \mathbb S_{d_i}$, $d_i=\deg f_i$, are epimorphisms and 
$f_{i\mid\, X_{\mathbb P^2}*}$ are the compositions of the homomorphsm $\iota_*$ an the monodromy homomorphisms 
$f_{i*}:\pi_1(\mathbb P^k\setminus B)\to \mathbb S_{d_i}$, $f_{i\mid\, X_{\mathbb P^2}*}=f_{i*}\circ \iota_*$. Therefore the covers $f_1$ and $f_2$ are equivalent if and only if the covers $f_{i\mid\, X_{\mathbb P^2}}$ and $f_{i\mid\, X_{\mathbb P^2}}$ are equivalent.

To compleat the proof of Theorem 3 it suffices to note that in the case of $k=2$, Theorem 3(1) is proved in \cite{N}  and the Theorem 3(2) is proved in \cite{K2}.

\ifx\undefined\bysame
\newcommand{\bysame}{\leavevmode\hbox to3em{\hrulefill}\,}
\fi


\begin{thebibliography}{99}

\bibitem{Ch} { O. Chisini:} {\it Sulla identita birazionale delle funzioni algebriche di due variabili dot ate di una
medesima curva di diramazione,}  Rend. 1st. Lombardo, {\bf 77} (1944),  339 -- 356.

\bibitem{C-F} {C. Ciliberto, F. Flamini:} {\it On the branch curve of a generic projection of a surface to a plane,}
Trans. Amer. Math. Soc., {\bf 363:7} (2011), 3457 -- 3471.

\bibitem{K2} Vik.S. Kulikov: {\it On Chisini's conjecture. II,} Izv. Math., {\bf 72:5} (2008), 63 -- 76.

\bibitem{K3} {Vik.S. Kulikov:} {\it Generalized Chisini's Conjecture,} Proc. Steklov Inst. Math., {\bf 241} (2003), 110 -- 119.

\bibitem{Ku} Vik.S. Kulikov: {\it On the structure of the fundamental group of the complement of algebraic curves in 
$\mathbb C^2$}, Russian Acad. Sci. Izv. Math., {\bf 40:2} (1993), 443 -- 454. 

\bibitem{K_G} {Vik.S. Kulikov:} {\it On the  almost generic covers of the projective plane, } Pure and Appl. Math. Quarterly, {\bf 16:4} (2020), 1067 -- 1082.

\bibitem{N} S.Yu. Nemirovski: {\it Kulikov's theorem on the Chisini conjecture,} Izv. Math., {\bf 65:1} (2001), 71 -- 74.

\bibitem{Sh} {I.R. Shafarevich:} {\it Basic algebraic geometry,} (1972), 568 pp.   
\bibitem{St}{K. Stein:} {\it Analytische Zerlegungen komplexer R\"aume,} Math. Ann. {\bf 132} (1956), 63 -- 93.

\end{thebibliography}
 \end{document}